\documentclass[12pt,reqno]{amsart}
\usepackage{hyperref}
\usepackage{amsfonts,mathrsfs,bbm,rawfonts,amsmath,amssymb,amsthm}
\usepackage{fullpage, setspace}
\usepackage{color}
\usepackage{galois}

\numberwithin{equation}{section}
\allowdisplaybreaks

\def\endproof{$\hfill\Box$\\}
\def\s{\,\,\,\,}
\def\R{\mathbb{R}}

\numberwithin{equation}{section}
\newtheorem{theorem}{Theorem}[section]
\newtheorem{lem}[theorem]{Lemma}
\newtheorem{thm}[theorem]{Theorem}
\newtheorem{pro}[theorem]{Proposition}
\newtheorem{cor}[theorem]{Corollary}
\newtheorem{defi}[theorem]{Definition}
\newtheorem{rem}[theorem]{Remark}

\def\v{\mbox{V}}
\newcounter{Cnumber}

\title[ ]
{\bf \bf Yamabe Equation on Some Complete Noncompact Manifolds}
\author[ ]{Guodong Wei}
\thanks {The author is supported by NSFC No. 11471316.}

\begin{document}
\maketitle

\begin{abstract}
In this paper, we consider the Yamabe equation on a complete noncompact Riemannian manifold and find some geometric conditions on the  manifold such that the Yamabe problem admits a bounded positive solution.

\end{abstract}

\section{Introduction}
Let $(M,g)$ be a complete Riemannian manifold of dimension $n\geq3$. The Yamabe problem on a compact Riemannian manifold without boundary consists of finding a constant scalar curvature metric $\widetilde{g}$  which is pointwise conformally related to $g$. It is well known that this problem is equivalent to showing the existence of a positive solution to the equation
$$\Delta_{g} u-\frac{n-2}{4(n-1)}R_{g}u+Ku^{\frac{n+2}{n-2}}=0,$$
if one sets $\widetilde{g}=u^{\frac{4}{n-2}}g$.
Usually, one writes this equation as
\begin{equation}\label{Yamabe equation}
\Delta_{g} u-c(n)R_{g}u+Ku^{p-1}=0\s\s \mbox{on}\s  M,
\end{equation}
where $\Delta_{g}$ is the Laplace-Beltrami operator associated with $g$, $R_{g}$ is the scalar curvature of $g$, $c(n)=\frac{n-2}{4(n-1)}$, $p=\frac{2n}{n-2}$, and $K$ is a constant satisfying $K=c(n)R_{\widetilde{g}}$ where $R_{\widetilde{g}}$ is the scalar curvature of $\widetilde{g}$. As is well known, the existence of minimizing solution to the Yamabe problem on a compact manifold was established through the combined works of Yamabe \cite{Yamabe}, Trudinger \cite{Trudinger}, Aubin \cite{Aubin2} and Schoen \cite{Schoen}.
\medskip

In the 1980's, Yau \cite{Yau} and Kazdan \cite{Kazdan} suggested the study of \eqref{Yamabe equation} in a noncompact complete manifold. In the book \cite{Aubin2}, this study was proposed again by Aubin. For the case $(M, g)$ is noncompact complete manifold with nonpositive scalar curvature, Aviles and McOwen have ever established some existence results in \cite{Aviles-McOwen}. However, the understanding on the case $(M, g)$ is of nonnegative scalar curvature is still rather limited. Some existence and nonexistence results on the case $(M,g)$ is of positive scalar curvature have been established in \cite{Kim2}, \cite{Zhang} and \cite{Jin}. In \cite{Kim1} and \cite{Kim2}, S. Kim introduced a functional $Y_{\infty}(M)$ (which may be called the Yamabe constant at infinity) to study the Yamabe problem on a complete noncompact manifold and got an existence result merely under the assumption $Y(M)<Y_{\infty}(M)$ for such manifold $(M, g)$ with positive scalar curvature. However, Zhang \cite{Zhang} found a gap in Kim's proof, and fixed the gap under an additional assumption on the volume growth of geodesic balls.

In this paper, we focus on the solvability of the Yamabe problem on a complete noncompact Riemannian manifold without the nonnegativity assumption on its scalar curvature, and intend to improve and generalize some results obtained in \cite{Zhang} and \cite{Kim1}. More concretely, firstly we try to prove a similar existence result with that in \cite{Zhang} under a weaker assumptions on the volume growth of geodesic balls except for the nonnegativity of the scalar curvature of $(M, g)$ is not be assumed. Then, we try to replace the hypothesis on the volume growth of geodesic balls of $M$ by some other geometric hypotheses to derive some existence results.

In order to state our results, we need to clarify some notations. Generally, $(M, g)$ denotes a complete noncompact manifold with $\mbox{dim} (M) \geq 3$. $O$ is a fixed point in $M$ and $d(x)=d(x,O)$ denotes the distance from $O$ to any $x\in M$ with respect to $g$, $R_{m}$ is the curvature tensor and $\v_{g}(B_{r}(x))$ denotes the volume of $B_{r}(x)$. Let $Y(M)$ and $Y_{\infty}(M)$ be the Yamabe constant and the Yamabe constant at infinity on $M$ respectively.

Throughout this article, we always assume $M$ satisfies the following conditions:
\begin{equation}\label{condition}
Y(M)<Y_{\infty}(M)\s\s \mbox{and}\s\s Y_{\infty}(M)>0.
\end{equation}
It is worthy to point out that $Y_{\infty}(M)>0$ implies that $Y(M)>-\infty$ by Theorem 1.7 in \cite{Groe-Nardmann}. Now we are ready to state our main results.
\begin{thm}\label{mainthm1}
Let $(M,g)$ be a complete noncompact Riemannian manifold of dimension $n\geq3$ and satisfy \eqref{condition}. Suppose that there exists a positive constant $C$ such that $R_{g}\geq -C d(x)^{-2}$ when $d(x)$ is large . Then there exists a constant $\rho_{0}\left(n,\ Y(M),\ Y_{\infty}(M)\right)>0$ such that, if $\v_{g}(B(O,r))\leq Cr^{n+\rho}$ for all large $r$ where $\rho$ is a number with $\rho<\rho_{0}$, then the Yamabe equation \eqref{Yamabe equation} admits a positive solution $u$ with $K=1,0,-1$ corresponding to $Y(M)$ is positive, $0$ and negative respectively. Moreover, we have
$$\lim_{d(x)\rightarrow \infty}d(x)^{\alpha}u(x)=O(1),$$
where $\alpha=\alpha(n,\rho)>0.$
\end{thm}
The method we used here is inspired by Zhang \cite{Zhang} and S.Kim \cite{Kim2}. Under the control condition on volume growth stated in the above theorem, we can derive a priori decay estimate of the `approximate solutions $\{u_{i}\}$' (see the detail in Theorem \ref{solution}). Hence, it follows that, if the sequence $\{u_{i}\}$ blows up, then the blow up points must lie in a compact subset of $M$.
\begin{rem}
In the case $\rho<0$ and the scalar curvature of $(M, g)$ is nonnegative, we have $${\lim_{d(x)\to\infty}d(x)^{\frac{n-2}{2}}u(x)=o(1)},$$
which is just the main result in Q. Zhang \cite{Zhang}.
\end{rem}
\begin{rem}
There are a lot of manifolds satisfying the condition $Y(M)<Y_{\infty}(M)$, such as $M$ is not locally conformally flat and there exists a compact subset $M_{0}$ such that $M\backslash M_{0}$ admits a conformal map to $S^{n}$ (see Schoen and Yau \cite{Schoen-Yau} and S.Kim \cite{Kim2}). In \cite{Zhang}, Q.Zhang constructed a explicit example on warped product manifold.
\end{rem}

By the volume comparison theorem, the following corollary is an immediate consequence of Theorem \ref{mainthm1}.
\begin{cor}
Let $(M,g)$ be a complete noncompact Riemannian manifold of dimension $n\geq3$ with nonnegative Ricci curvature. Suppose $ Y(M)<Y_{\infty}(M)$. Then the Yamabe equation admits a positive solution $u$ with $K=1$ and
$$\lim_{d(x)\rightarrow \infty}d(x)^{\frac{n-2}{2}}u(x)=O(1).$$
\end{cor}
It is natural to ask what happens without the assumption on volume growth? In this situation, one will encounter a new difficulty that, if $\{u_{i}\}$ blows up, maybe the blow up points tend to infinity of $M$. To overcome this difficulty, we need to analyze the convergence of the pointed manifolds induced by $\{u_{i}\}$ under the pointed Cheeger-Gromov topology by providing certain suitable conditions, then we discuss the blowup behavior of $u_{i}$ on the limit pointed manifold. We obtain the following results.
\begin{thm}\label{mainthm2}
Let $(M,g)$ be a complete noncompact Riemannian manifold of dimension $n\geq3$ and satisfy \eqref{condition}. Assume
$|\mbox{Ric}(M)|\leq C_{0}$ and $\mbox{inj}(M)\geq i_{0}$ where $C_{0}$ and $i_{0}$ are positive constants. Then \eqref{Yamabe equation} has a positive solution with $K=1,0,-1$ corresponding to $Y(M)$ is positive, $0$ and negative respectively. Moreover, there holds true
$$\lim_{d(x)\rightarrow \infty}u(x)=O(1).$$
\end{thm}

If we do not have a priori positive lower bound of the injective radius, by Anderson \cite{Anderson1} and \cite{Anderson2} we can also get the following conclusion .
\begin{thm}\label{mainthm3}
Let $(M,g)$ be a complete noncompact Riemannian manifold of dimension $n\geq3$ and satisfy \eqref{condition}. Suppose
\begin{itemize}
\item[(i)] there exist positive constants $C_{0}$, $v_{0}$ and $C$ such that $Ric(M)\leq C_{0}$, $\v_{g}(B(p,1))\geq v_{0}$ for any $p\in M$, and $\int_{M}|R_{m}|^{\frac{n}{2}}d\v_{g}\leq C$ respectively;
\item[(ii)] the length of the shortest inessential (null-homotopic) geodesic loops, denoted by $l_M$, is positive, i.e. $l_{M}\geq l>0$, or $M$ is odd-dimensional, oriented manifold.
\end{itemize}
Then, the Yamabe equation \eqref{Yamabe equation} has a positive solution.
\end{thm}

The paper is organized as follows: In section \ref{preliminary}, we recall some basic notations, prove some basic facts about Yamabe functional and discuss the variational approach as in the compact case. In section \ref{proof1}, we give the proof of Theorem \ref{mainthm1}. In section \ref{proof2} we prove the Theorem \ref{mainthm2} and \ref{mainthm3} by analyzing the blowup behavior of $\{u_{j}\}$ under the pointed Cheeger-Gromov topology.

\medskip
\section{Some basic notations and known results}\label{preliminary}
In this section, we will recall some basic notations and definitions such as the Yamabe constants $Y(M)$ and $Y_{\infty}(M)$. Then we discuss the existence of `smooth approximate solutions $u_{i}$' corresponding to the exhaustion of $M$. The main methods and techniques used in this section can be found in the survey paper \cite{Lee-John}.  For the sake of clarity and completeness, we shall still write it down. At last, we recall the definition of pointed Cheeger-Gromov topology.

\subsection{Yamabe constant on noncompact manifold}

For any $\ \upsilon \in C_{c}^{\infty}(M)\backslash\{0\}$, define
$$
E_{g}(\upsilon)=\  \int_{M}(|\nabla\upsilon|^{2}+c(n)R_{g}\upsilon^{2})d\v_{g},
$$
then Yamabe constant of $(M,g)$ is defined by
$$
Y(M)=\inf\left\{\frac{E_{g}(\upsilon)}{\|\upsilon\|_{L^{p}(g)}^{2}}|\upsilon \in C_{c}^{\infty}(M)\backslash\{0\}\right\}.
$$
In \cite{Kim1} and \cite{Kim2}, S.Kim defined a new functional called the Yamabe constant at infinity for noncompact manifold as follow:
Choose an exhaustion $\{K_{i}\}_{i\in \mathbb{N}}$ of $M$, which is composed of bounded set, and define
$$
Y_{\infty}(M)=\lim_{i\rightarrow\infty}Y(M\backslash K_{i}).
$$
Obviously $Y_{\infty}(M)$ does not depend on the exhaustion we choose.
\begin{lem}\label{lem2.1}
For any complete non-compact manifold $M$, there always holds
$$ -c(n)\|(R_{g})_{-}\|_{L^{\frac{n}{2}}}\leq Y(M)\leq Y_{\infty}(M)\leq \Lambda,$$
where $\Lambda=\frac{n(n-2)}{4}\omega_{n}^{\frac{2}{n}}$ is the best Sobolev constant on $\R^{n}$ and $ (R_{g})_{-}$ is the negative part of the scalar curvature on $M$.
\end{lem}
\begin{proof}
The first inequality is derived by H\"{o}lder inequality and the second holds evidently by the definition of $Y(M)$ and $Y_{\infty}(M)$.

In order to prove the inequality on the right hand side, we need to take the following arguments.  Let $$u_{\alpha}=\left(\frac{\alpha}{\alpha^{2}+|x|^{2}}\right)^{\frac{n-2}{2}}.$$
It is well-known that we may obtain the best Sobolev constant in $\R^{n}$ by this family $\{u_{\alpha}\}$. In other words, there holds
$$\int_{\R^{n}}|\nabla u_{\alpha}|^{2}dx=\Lambda\left(\int_{\R^{n}}u_{\alpha}^{p}\ dx\right)^{\frac{2}{p}}.$$
For any $q\in M\backslash K_{i}$, we choose the normal coordinates around $q$. It is well-know that, in the normal coordinates, there holds $d\v_{g}=(1+O(r))$. Given $\epsilon>0$, let $B_{\epsilon}$ denote the ball of radius $\epsilon$ in $\R^{n}$. We choose a smooth radial cutoff function $0\leq\eta(r)\leq1$ which is supported in $B_{2\epsilon}$ and $\eta\equiv1$ on $B_{\epsilon}$. Setting $\varphi=\eta u_{\alpha}$, we have
\begin{align}\label{laowei}
\int_{\R^{n}}|\nabla \varphi|^{2}\ dx&=\int_{B_{2\epsilon}}(\eta^{2}|\nabla u_{\alpha}|^{2}+2\eta u_{\alpha}\langle\nabla\eta,\nabla u_{\alpha}\rangle+u_{\alpha}^{2}|\nabla\eta|^{2})\ dx\nonumber\\
&\leq\int_{\R^{n}}|\nabla u_{\alpha}|^{2}\ dx +C\int_{A_{\epsilon}}(u_{\alpha}|\nabla u_{\alpha}|+u_{\alpha}^{2})\ dx,
\end{align}
where $A_{\epsilon}$ denotes the annulus $B_{2\epsilon}\backslash B_{\epsilon}$. Since $$u_{\alpha}\leq\alpha^{\frac{n-2}{2}}r^{2-n}\s\s\s \mbox{and}\s\s\s  |\nabla u_{\alpha}|\leq(n-2)\alpha^{\frac{n-2}{2}}r^{1-n},$$
then, for fixed $\epsilon$, the second term on the right hand side of the above inequality is $O(\alpha^{n-2})$ as $\alpha\rightarrow0$.
For the first term of (\ref{laowei}), we have
\begin{align*}
\int_{\R^{n}}|\nabla u_{\alpha}|^{2}\ dx&=\Lambda\left(\int_{B_{\epsilon}}u_{\alpha}^{p}\ dx+\int_{\R^{n}\backslash B_{\epsilon}}u_{\alpha}^{p}\ dx\right)^{\frac{2}{p}}\\
&\leq\Lambda\left(\int_{B_{2\epsilon}}\varphi^{p}\ dx+\int_{\R^{n}\backslash B_{\epsilon}}\alpha^{n}r^{-2n}\ dx\right)^{\frac{2}{p}}\\
&\leq\Lambda\left(\int_{B_{2\epsilon}}\varphi^{p}\ dx\right)^{\frac{2}{p}}+O(\alpha^{n-2}).
\end{align*}
Therefore, on $M$, we have the following:
\begin{align*}
&\int_{B_{2\epsilon}}(|\nabla \varphi|^{2}+c(n)R_{g}\varphi^{2})\ d\v_{g}\\
\leq&\left(1+C\epsilon\right)\left(\Lambda\|\varphi\|_{L^{p}}^{2}+C\alpha^{n-2}+C\int_{0}^{2\epsilon}\int_{S_{r}}u_{\alpha}^{2}r^{n-1}\ d\omega dr\right).
\end{align*}
The last term on the right hand side of the above inequality is actually bounded by a constant multiple of $\alpha$. Obviously,
\begin{equation*}
\int_{0}^{2\epsilon}u_{\alpha}^{2}r^{n-1}\ dr=\alpha^{2}\int_{0}^{\frac{2\epsilon}{\alpha}}\sigma^{n-1}(\sigma^{2}+1)^{2-n}\ d\sigma,
\end{equation*}
noting that $\sigma^{2}\leq\sigma^{2}+1\leq 2\sigma^{2}$ for $\sigma\geq1$ we can see that there holds true
$$C_1\left(C+\alpha^{2}\int_{1}^{\frac{2\varepsilon}{\alpha}}\sigma^{3-n}\ d\sigma\right)\leq\alpha^{2}\int_{0}^{\frac{2\epsilon}{\alpha}}\sigma^{n-1}(\sigma^{2}+1)^{2-n}\ d\sigma\leq C_2\left(C+\alpha^{2}\int_{1}^{\frac{2\varepsilon}{\alpha}}\sigma^{3-n}\ d\sigma\right).$$
A simple computation shows
$$\alpha^{2}\int_{1}^{\frac{2\varepsilon}{\alpha}}\sigma^{3-n}\ d\sigma\leq \left\{
\begin{aligned}
\alpha\ \ \ \mbox{ if } n=3,\\
-\alpha^{2}\log\alpha\ \ \ \mbox{ if } n=4,\\
\alpha^{2}\ \ \ \mbox{ if } n\geq5.
\end{aligned}
\right.$$
Thus, choosing first $\epsilon$ and then $\alpha$ small, we can arrange that
$$\frac{E_{g}(\varphi)}{\|\varphi\|_{L^{p}}^{2}}\leq (1+C\epsilon)(\Lambda+C\alpha).$$
Since $\epsilon$ and $\alpha$ can be arbitrarily small,  it follows that
$$\frac{E_{g}(\varphi)}{\|\varphi\|_{L^{p}}^{2}}\leq \Lambda.$$
Thus, we complete the proof.
\end{proof}

It is worthy to point out that we do not assume that the injective radius of $M$  is of the positive lower bound in the proof of the above lemma.

\subsection{The variational approach}
In this subsection, let $B_{r}(O)$ denote the geodesic ball centered at $O$ with radius $r$ on $M$ ($M$ is noncompact), where $O$ is a fixed point in $M$. Denote
$$Y_{j}=\inf_{\phi\in W_{0}^{1,2}(B_{j}(O))\backslash\{0\}}\left\{\frac{E_{g}(\phi)}{\|\phi\|_{L^{p}(g)}^{2}}\right\}$$
We have the following proposition:
\begin{pro}\label{solution}
Assume $Y_{j}<\Lambda$. Then, the following Dirichlet problem admits a positive solution $u_{j}$ with $\|u_j\|_{L^{p}}=1$
\begin{eqnarray}
\Delta u_{j}-c(n)R_{g}u_{j}+Y_{j}u_{j}^{p-1}&=&0,\s \mbox{in} \s B_{j}(O),\\
u_{j}&=&0,\s \mbox{on} \s \partial B_{j}(O).
\end{eqnarray}
\end{pro}

To prove this proposition, we need to establish some lemmas. Firstly, for $s\in(2,p]$, we define
$$Q_{s}(u)=\frac{E_{g}(u)}{\|u\|_{L^{s}(g)}^{2}}$$
and
$$\lambda_{s}=\inf\{Q_{s}(u)|u\in W_{0}^{1,2}(B_{j}(O))\backslash\{0\}\}.$$

\begin{lem}(\cite{Aubin1}\label{aubinlemma}) For $Q_s(u)$ and $\lambda_s$ defined as above, there always holds
$$\limsup_{s\rightarrow p}\lambda_{s}\leq Y_{j}.$$
Moreover,  if $\lambda_{s}\geq 0$, then $\lambda_{s}\rightarrow Y_{j}$.
\end{lem}

\begin{lem}
For any $s\in(2,p)$, there exists $u_{s}\in C^{\infty}(\overline{B}_{j}(O))$, $u_{s}>0$ in $B_{j}(O)$, $u_{s}=0$ on $\partial B_{j}(O)$ and $\|u_{s}\|_{L^{s}}=1$ such that
$Q_{s}(u_{s})=\lambda_{s}$ and satisfying the following equation:
\begin{equation}\label{eq1}
\Delta u_{s}-c(n)R_{g}u_{s}+\lambda_{s} u_{s}^{s-1}=0.
\end{equation}
\end{lem}

\begin{proof}
Take a minimizing sequence $\{u_{i}\}\subset  W_{0}^{1,2}(B_{j}(O))\backslash\{0\}$ s.t $Q_{s}(u_{i})\rightarrow \lambda_{s}$. Since $Q_{s}(|u|)\leq Q_{s}(u)$ and $Q_{s}(tu)=Q_{s}(u)$ , we can assume $u_{i}\geq 0$ and $\|u_{i}\|_{L^{s}}=1$. Then we have
$$Q_{s}(u_{i})=E_{g}(u_{i})=\|\nabla u_{i}\|^{2}_{2}+c(n)\int_{B_{j}(O)}R_{g}u_{i}^{2}d\v_{g}\rightarrow \lambda_{s}.$$
Hence we have $\|\nabla u_{i}\|^{2}_{L^{2}}\leq c_{1}+c_{2}\|u_{i}\|^{2}_{L^{2}}$. By H\"{o}lder inequality, we also have $$\|u_{i}\|^{2}_{L^{2}}\leq C(\v_{g}(B_j(O))\|u_{i}\|_{L^{s}}^{2}= C(\v_{g}(B_j(O)).$$
Therefore, $\{u_{i}\}$ is a bounded sequence in $W_{0}^{1,2}(B_{j}(O)$. Then, neglecting a subsequence, there exists $u_{s}\in W_{0}^{1,2}(B_{j}(O))$ such that $\{u_{i}\}$ converges weakly to $u_{s}$ in $W^{1,2}(B_{j}(O))$. On the other hand side, we also know that $W^{1,2}\hookrightarrow L^{r}$ is compactly embedded when $0\leq r< p $. Hence we have
\begin{eqnarray}
\|\nabla u_{s}\|_{L^{2}}\leq\liminf_{i\rightarrow\infty}\|\nabla u_{i}\|_{L^{2}},\\
\int R_{g}u_{i}^{2}\ d\v_{g}\rightarrow\int R_{g}u_{s}^{2}\ d\v_{g},\\
\|u_{s}\|_{L^{s}}=\lim_{i\rightarrow\infty}\|u_{i}\|_{L^{s}}=1.
\end{eqnarray}
Combining the above three inequalities we infer
$$Q_{s}(u_{s})\leq \liminf_{i\rightarrow\infty}Q_{s}(u_{i})=\lambda_{s}.$$
Then, the definition of $\lambda_{s}$ tells us $Q_{s}(u_{s})=\lambda_{s}$. This means that $u_{s}$ is the weak solution of \eqref{eq1}. Using $L^{p}$ estimate and Schauder estimate, we take a standard boot-strapping argument to deduce $u_{s}\in C^{2,\alpha}(\overline{B}_{j}(O))$.

Since $u_{i}\geq0$, it follows that $u_{s}\geq0$. Hence, it is easy to see that there exist some constant $c\geq0$ such that $\Delta u_{s}-cu_{s}\leq 0$. By the maximal principle, we have $u_{s}>0$ in $B_{j}(O)$. Since $t^{s-1}$ is a smooth function when $t>0$, it follows that $u_s^{s-1}$ is a smooth function. Hence, the standard elliptic theory tells us that $u_{s}\in C^{\infty}(\overline{B}_{j}(O))$.
\end{proof}

\begin{lem}\label{lem2.5}
$\{u_{s}|s_{0}\leq s<p\}$ is uniformly bounded with respect to $s$ for some constant $s_{0}\in (2,\, p)$.
\end{lem}
\begin{proof}
Since each $u_{s}$ satisfies the equation \eqref{eq1} and $u_{s}=0$ on $\partial B_{j}(O)$. Let $b>0$ be a constant which will be determined later. Multiplying \eqref{eq1} both side by $u_{s}^{1+2b}$ and integrating by parts, we obtain
$$\int_{B_{j}(O)}\left(\langle\nabla u_{s},(1+2b)u_{s}^{2b}\ \nabla u_{s}\rangle+c(n)R_{g}u_{s}^{2+2b}\right)\ d\v_{g}=\lambda_{s}\int_{B_{j}(O)}u_{s}^{s+2b}\ d\v_{g}.$$
If we set $w=u_{s}^{1+b}$, then the above equality can be written as
$$ \frac{1+2b}{(1+b)^{2}}\int_{B_{j}(O)}|\nabla w|^{2}d\v_{g}=\int_{B_{j}(O)}\left(\lambda_{s}w^{2}u_{s}^{s-2}-c(n)R_{g}w^{2}\right)d\v_{g}.$$
Now, applying the sharp Sobolev inequality, for any $\epsilon>0$, there exists some $C(\epsilon)$ such that
\begin{align}\label{S:1}
\|w\|_{L^{p}}^{2}&\leq\frac{(1+\epsilon)}{\Lambda}\int_{B_{j}(O)}|\nabla w|^{2}d\v_{g}+C(\epsilon) \int_{B_{j}(O)}w^{2}d\v_{g}\nonumber\\
&\leq(1+\epsilon)\frac{(1+b)^{2}}{1+2b}\int_{B_{j}(O)}\frac{\lambda_{s}}{\Lambda}w^{2}u_{s}^{s-2}d\v_{g}+C'(\epsilon)\|w\|_{L^{2}}^{2}\nonumber\\
&\leq(1+\epsilon)\frac{(1+b)^{2}}{1+2b}\frac{\lambda_{s}}{\Lambda}\|w\|_{L^{p}}^{2}\|u_{s}\|_{L^{n(s-2)/2}}^{s-2}+C'(\epsilon)\|w\|_{L^{2}}^{2}.
\end{align}
Since $s<p$, we have $(s-2)n/2<s$. By H\"{o}lder inequality, we have $$\|u_{s}\|_{L^{n(s-2)/2}}\leq C(s)\|u_{s}\|_{L^{s}}=C(s),$$ where $C(s)\rightarrow 1$ as $s$ tends to $p$.

Now, we need to consider the following two cases:

 Case 1: $0\leq Y_{j}<\Lambda$. In this case we have $\lambda_{s}\geq 0$. Moreover, by Lemma \ref{aubinlemma} we know that there exists some $s_{0}\in(0, p)$ such that there holds $\lambda_{s}/\Lambda\leq\mu<1$ for any $s\in [s_{0}, p)$. Thus, we can choose $\epsilon$ and $b$ small enough such that the coefficient of the first term above
$$(1+\epsilon)\frac{(1+b)^{2}}{1+2b}\frac{\lambda_{s}}{\Lambda}<1.$$
So, it follows from (\ref{S:1}) that
$$\|w\|_{L^{p}}^{2}\leq C\|w\|_{L^{2}}^{2}.$$

Case 2: $Y_{j}<0$. For this case the same result holds obviously. Indeed, as $Y_{j}$ is less than zero, it follows Lemma 2.3 that $\lambda_{s}$. We apply the H\"{o}lder inequality to derive
$$\|w\|_{L^{2}}=\|u_{s}\|_{L^{2(1+b)}}^{1+b}\leq C\|u_{s}\|_{L^{s}}^{1+b}\leq C.$$
Therefore, we have $\|w\|_{L^{p}}=\|u_{s}\|_{L^{p(1+b)}}^{1+b}$ is bounded uniformly with respect to $s$. By $L^{p}$ estimates and Sobolev embedding theorem, we know that the lemma is true.
\end{proof}

\noindent {\bf Proof of Proposition \ref{solution}.}
By Lemma \ref{lem2.5}, we know $u_{s}$ is uniform bounded in $C^{k,\alpha}(\overline{B}_{j}(O))$. Hence, there exists a subsequence of $\{u_s\}$ which converges to a solution of $(2.1)$ and $(2.2)$.

\medskip

\subsection{Pointed Cheeger-Gromov topology} At the last part of this section, we  recall the definition of convergence of manifolds under the pointed Cheeger-Gromov topology.\begin{defi}\label{cheeger}
A sequence of pointed complete Riemann manifolds is said to converge in pointed $C^{m,\alpha}$ Cheeger-Gromov topology $(M_{i},p_{i}, g_{i})\rightarrow(M,p,g)$ if for every $R>0$ we can find a domain $B_{R}(p)\subset\Omega\subset M$ and embeddings $F_{i}:\Omega\rightarrow M_{i}$ for large $i$ such that $F_{i}(\Omega)\supset B_{R}(p_{i})$ and $F_{i}^{*}g_{i}\rightarrow g$ on $\Omega$ in the $C^{m,\alpha}$ topology.
\end{defi}
Note that $C^{m,\alpha}$ type convergence implies pointed  Gromov-Hausdorff convergence.

\section{Proof of theorem \ref{mainthm1}}\label{proof1}
 We proceed now to the proof of Theorem \ref{mainthm1}. Its proof will be divided into four steps. The basic idea we used here is to employ the finite domain exhaustion of $M$ and then consider the subsolution sequence $u_{i}$ of Yamabe equations corresponding to this exhaustion. A crucial step is to establish a decay estimate of $u_{i}$ near infinity.
\subsection{Step 1.}
 By condition $Y(M)<Y_{\infty}(M)$ and  Lemma \ref{lem2.1}, we have $Y(M)<\Lambda$. On the other side, by the definition of $Y_{j}$, we know that $\{Y_{j}\}$ converges decreasingly to $Y(M)$. So, when $j$ is large enough, we have $Y_{j}<\Lambda$. Using Theorem \ref{solution}, we know there is a positive solution $u_{j}$ solving the following equation:
\begin{eqnarray*}
\Delta u_{j}-c(n)R_{g}u_{j}+Y_{j}u_{j}^{p-1}&=&0,\s \mbox{in} \s B_{j}(O),\\
u_{j}&=&0,\s \mbox{on} \s \partial B_{j}(O).
\end{eqnarray*}
Next, we extend $u_{j}$ to the whole manifold by defining $u_{j}(x)=0$ when $x\notin B_{j}(O)$. The extended function, we still denote it by $u_{j}$, is continuous and  a subsolution to the equation
$$\Delta u-c(n)R_{g}u+Y_{j}u^{p-1}=0,\s \mbox{on} \s M.$$
\subsection{Step 2.}\label{step2}
In this step, we will establish a priori decay estimate for $\{u_{j}\}.$
\begin{lem}\label{lem3.1}
Then there exists a $\rho_{0}(n,Y(M),Y_{\infty}(M))>0$, such that for any $\rho<\rho_{0}$, if $ \v_{g}(B(O,r))\leq Cr^{n+\rho}$ for all large r,, then we have
$$\lim_{d(x)\rightarrow \infty}\lim_{j\rightarrow \infty}d(x)^{\alpha}u_{j}(x)=O(1),$$
where $\rho$ can be negative and $\alpha=\alpha(\rho,n)>0$.
\end{lem}
\begin{proof}
Given $R>1$, first we fix a point $x_{0}\in M$ such that $d(x_{0})=2R^{2}$, then we scale the metric by $\widetilde{g}=g/R^{4}$. Let $d_{1}$, $\nabla_{1}$, $R_{\widetilde{g}}$, $\Delta_{1}$ and $d\v_{\widetilde{g}}$ be the corresponding distance, gradient, scalar curvature, Laplace-Beltrami operator and volume element with respect to the rescaled manifold $\left(M,\widetilde{g}\right)$. Define $v_{j}(x)=R^{n-2}u_{j}(x)$.
Since
$$\Delta u_{j}-c(n)R_{g}u_{j}+Y_{j}u_{j}^{p-1}\geq 0,\s on \s M.$$
A direct computation shows
\begin{eqnarray}
\Delta_{1}v_{j}-c(n)R_{\widetilde{g}}v_{j}+Y_{j}v_{j}^{p-1}=R^{n+2}\left(\Delta u_{j}-c(n)R_{g}u_{j}+Y_{j}u_{j}^{p-1}\right)\geq0,\label{eq3}\\
\int_{d_{1}(x_{0},x)\leq 1}v_{j}^{p}\ d\v_{\widetilde{g}}=\int_{d(x_{0},x)\leq R^{2}}u_{j}^{p}\ d\mbox{vol}_{g}\leq\int_{M}u_{j}^{p}\ d\v_{g}= 1.\label{eq4}
\end{eqnarray}

Take $\phi\in C^{\infty}[0,\infty)$ such that $0\leq\phi\leq1$ and $|\phi'(r)|\leq C$, which satisfies that $\phi(r)=1$, when $r\in [0,1/2]$; $\phi(r)=0$, when $r\in [1,\infty)$. Let $G(s)=s^{\beta}$ and define
$$ F(t)=\int_{0}^{t}G'(s)^{2}ds=\frac{\beta^{2}}{2\beta-1}t^{2\beta-1}.$$
By a simple computation, we see that, as $\beta>1$, there holds true
\begin{equation}\label{relation}
sF(s)\leq s^{2}G'(s)^{2}=\beta^{2}G(s)^{2} .
\end{equation}

Let $\eta(x)=\phi(d_{1}(x_{0},x))$, we know the support $\mbox{spt}(\eta^{2}F(v))\subseteq M\setminus B_{R^{2}/2}(O)$. We multiply \eqref{eq3} by $\eta^{2}F(v)$ and then integrate by parts to obtain that, for some $\epsilon>0$,
\begin{align*}
(1-\epsilon)&\|\int|\nabla_{1}v|^{2}G'(v)^{2}\eta^{2}d\v_{\widetilde{g}}\leq\beta^{2}\epsilon^{-1}\int|\nabla_{1}\eta|^{2}G(v)^{2}d\v_{\widetilde{g}}\\
&-\frac{\beta^{2}}{2\beta-1}\int c(n)R_{\widetilde{g}}G(v)^{2}\eta^{2}d\v_{\widetilde{g}}+Y_{j}\int v^{p-2}vF(v)d\v_{\widetilde{g}},
\end{align*}
where $\epsilon $ may be chosen arbitrarily small. By the condition $R_{g}\geq-C d^{-2}(x,O)$, we have that, when $R$ is large enough, $R_{\widetilde{g}}\geq -C$. Hence, from the above inequality it follows
\begin{align}\label{maineq}
 &\|\nabla_{1}(G(v)\eta)\|_{L^{2}}^{2}+\int c(n)R_{\widetilde{g}}(G(v)\eta)^{2}d\v_{\widetilde{g}}\nonumber\\
\leq& C\beta^{2}\|\nabla_{1}\eta\|_{L^{\infty}}^{2}\int G(v)^{2}d\v_{\widetilde{g}}+\frac{Y_{j}}{1-\epsilon}\int v^{p-1}F(v)\eta^{2}d\v_{\widetilde{g}}.
\end{align}

Next, we need to consider the following two cases:

\noindent {\bf Case 1:} $Y(M)\geq0$. Since $Y_{j}$ converges decreasingly to $Y(M)$, it follows that $Y_{j}\geq 0$. On the other side, by the assumption $Y_{\infty}(M)>0$ and the fact $Y(M\backslash B_{r}(O))$ increases with respect to $r$, we have $Y(M\backslash B_{r}(O)>0$ when $r$ is large enough. Let
$$C_{0}=\frac{1}{Y\left(M\backslash B_{\frac{R^{2}}{2}}(O)\right)}.$$
It is easy to see that $C_{0}>0$ as $R$ is large. Since $\mbox{spt}(\eta^{2}F(v))\subseteq M\backslash B_{R^{2}/2}(O)$, by the definition of the Yamabe quotient we have
$$\|G(v)\eta\|_{L^{p}}^{2}\leq C_{0}\|\nabla_{1}(G(v)\eta)\|_{L^{2}}^{2}+C_{0}\int_{d_{1}(x_{0},x)\leq1}c(n)R_{\widetilde{g}}[G(v)\eta]^{2}d\v_{\widetilde{g}}.$$
Here, all the norms were taken on the domain $d_{1}(x_{0},x)\leq1$ with respect to the rescaled metric $\widetilde{g}$. By the fact that $Y_{j}\geq 0$ and \eqref{relation}, we obtain
\begin{eqnarray}
\|G(v)\eta\|_{L^{p}}^{2}
&\leq& CC_{0}\beta^{2}\|\nabla_{1}\eta\|_{L^{\infty}}^{2}\int G(v)^{2}\ d\v_{\widetilde{g}}+C_{0}Y_{j}(\beta^{2}+\epsilon)\int v^{p-2}G(v)^{2}\eta^{2}\ d\v_{\widetilde{g}}\label{eq8}\\
&\leq& CC_{0}\beta^{2}\|\nabla_{1}\eta\|_{L^{\infty}}^{2}\int G(v)^{2}\ d\v_{\widetilde{g}}+C_{0}Y_{j}(\beta^{2}+\epsilon)\|G(v)\eta\|_{L^{p}}^{2}\left(\int v^{p}\ d\v_{\widetilde{g}}\right)^{\frac{2}{n}}\nonumber\\
&\leq&CC_{0}\beta^{2}\|\nabla_{1}\eta\|_{L^{\infty}}^{2}\int G(v)^{2}\ d\v_{\widetilde{g}}+C_{0}Y_{j}(\beta^{2}+\epsilon)\|G(v)\eta\|_{L^{p}}^{2}.\label{eq5}
\end{eqnarray}
Here we have used inequality \eqref{eq4} in the last inequality.

By the assumption $Y(M)<Y_{\infty}(M)$, we have, when $R$ sufficiently large,
$$Y(M)<Y\left(M\backslash B_{\frac{R^{2}}{2}}(O)\right).$$
Noting $Y_{j}\downarrow Y(M)$ as $j\rightarrow\infty$, we have that, when $j$ is large enough,
$$Y_{j}<Y\left(M\backslash B_{\frac{R^{2}}{2}}(O)\right)=\frac{1}{C_{0}}.$$
Hence, there exists $\beta_{0}>1$ such that, for all $j$ and small enough $\epsilon$,
\begin{equation}\label{eq6}
(\beta_{0}^{2}+\epsilon)C_{0}Y_{j}<1.
\end{equation}
Substitute \eqref{eq6} to \eqref{eq5} to obtain
\begin{equation}\label{S:2}
\|v^{\beta_{0}}\eta\|_{L^{p}}^{2}\leq C \|\nabla_{1}\eta\|_{L^{\infty}}^{2}\int v^{2\beta_{0}}d\v_{\widetilde{g}}.
\end{equation}
By the definition of $\eta$ and H\"{o}lder inequality, we infer from (\ref{S:2})
\begin{eqnarray}
\left(\int_{B_{1}(x_{0},\frac{1}{2})}v^{\frac{2n\beta_{0}}{n-2}}\ d\v_{\widetilde{g}}\right)^{\frac{n-2}{n}}&\leq&\|v^{\beta_{0}}\eta\|_{L^{p}}^{2}\leq C\int_{B_{1}(x_{0},1)}v^{2\beta_{0}}\ d\v_{\widetilde{g}}\nonumber\\
&\leq& C\left(\int_{B_{1}(x_{0},1)}v^{p}\ d\v_{\widetilde{g}}\right)^{\frac{\beta_{0}(n-2)}{n}}\left(\v_{\widetilde{g}}(B_{1}(x_{0},1))\right)^{\frac{n-(n-2)\beta_{0}}{n}}\nonumber\\
&\leq& C\left(\v_{\widetilde{g}}(B_{1}(x_{0},1))\right)^{\frac{n-(n-2)\beta_{0}}{n}}.
\end{eqnarray}

Now, we proceed to a consideration of the possible growth rate of volume of geodesic ball such that the Yamabe equation \eqref{Yamabe equation} on $M$ is solvable. If $\v_{g}\left(B(O,r)\right)\leq Cr^{n+\rho}$ for all $r$ large, then we have
$$\v_{\widetilde{g}}(B_{1}(x_{0},1)))=\frac{\v_{g}\left(B_{R^{2}}(x_{0})\right)}{R^{2n}}\leq\frac{\v_{g}\left(B_{4R^{2}}(O)\right)}{R^{2n}}\leq CR^{2\rho}.$$
Therefore, we obtain the following
\begin{equation}\label{eq7}
\left(\int_{B_{1}(x_{0},\frac{1}{2})}v^{\frac{2n\beta_{0}}{n-2}}d\v_{\widetilde{g}}\right)^{\frac{n-2}{n}}\leq CR^{\frac{2\rho[n-(n-2)\beta_{0}]}{n}}.
\end{equation}
Subsequently we will use a standard Moser iteration argument to finish the proof the lemma.
Given $0<r_{2}<r_{1}<\frac{1}{2}$, by taking $G(v)=v^{\beta}$ we have
\begin{eqnarray*}
\int_{B_{1}(x_{0},r_{1})}v^{p-2}G(v)^{2}\ d\v_{\widetilde{g}}&\leq&\left(\int_{B_{1}(x_{0},r_{1})}v^{\frac{2n\beta_{0}}{n-2}}\ d\v_{\widetilde{g}}\right)^{\frac{2}{n\beta_{0}}}\left(\int_{B_{1}(x_{0},r_{1})}G(v)^{\frac{2n\beta_{0}}{n\beta_{0}-2}}\ d\v_{\widetilde{g}}\right)^{\frac{n\beta_{0}-2}{n\beta_{0}}}\\
&\leq&\left(\int_{B_{1}(x_{0},r_{1})}v^{\frac{2n\beta_{0}}{n-2}}\ d\v_{\widetilde{g}}\right)^{\frac{2}{n\beta_{0}}}\left(\int_{B_{1}(x_{0},r_{1})}G(v)^{\frac{2n\delta}{n-2}}\ d\v_{\widetilde{g}}\right)^{\frac{n-2}{n\delta}},
\end{eqnarray*}
where
$$\delta=\frac{(n-2)\beta_{0}}{n\beta_{0}-2}<1.$$
Therefore, by combining \eqref{eq7} with the above inequality we have
\begin{equation}\label{eq9}
\int_{B_{1}(x_{0},r_{1})}v^{p-2}G(v)^{2}d\v_{\widetilde{g}}\leq CR^{\frac{4\rho}{n-2}\frac{n-(n-2)\beta_{0}}{n\beta_{0}}}\left(\int_{B_{1}(x_{0},r_{1})}G(v)^{\frac{2n\delta}{n-2}}d\v_{\widetilde{g}}\right)^{\frac{n-2}{n\delta}}.
\end{equation}

Noting that
$$\frac{n\delta}{n-2}=\frac{n\beta_{0}}{n\beta_{0}-2}>1,$$
we use again the H\"{o}lder inequality to obtain
\begin{eqnarray}
\int_{B_{1}(x_{0},r_{1})}G(v)^{2}\ d\v_{\widetilde{g}}&\leq&\left(\int_{B_{1}(x_{0},r_{1})}G(v)^{\frac{2n\delta}{n-2}}\ d\v_{\widetilde{g}}\right)^{\frac{n-2}{n\delta}}\ \v_{\widetilde{g}}\left(B_{1}(x_{0},1)\right)^{\frac{n\delta-n+2}{n\delta}}\nonumber\\
&\leq& C\left(\int_{B_{1}(x_{0},r_{1})}G(v)^{\frac{2n\delta}{n-2}}\ d\v_{\widetilde{g}}\right)^{\frac{n-2}{n\delta}}R^{\frac{4\rho}{n\beta_{0}}}\label{eq10}.
\end{eqnarray}

For $0<r_{2}<r_{1}<\frac{1}{2}$, we choose $\eta$ to be a radial function, supported in $B_{1}(x_{0},r_{1})$, such that $\eta=1$ if $x\in B_{1}(x_{0},r_{2})$ and $|\nabla_{1}\eta|\leq \frac{2}{r_{1}-r_{2}}$. We also note that \eqref{eq8} remains valid for such $\eta$ and any fixed $\beta>1$,i.e.
\begin{eqnarray}\label{S:3}
\|G(v)\eta\|_{L^{p}}^{2}
\leq CC_{0}\beta^{2}\|\nabla_{1}\eta\|_{L^{\infty}}^{2}\int G(v)^{2}\ d\v_{\widetilde{g}}+C_{0}Y_{j}(\beta^{2}+\epsilon)\int v^{p-2}G(v)^{2}\eta^{2}\ d\v_{\widetilde{g}}.
\end{eqnarray}
Substituting \eqref{eq9} and \eqref{eq10} into the right hand side of (\ref{S:3}), we obtain
\begin{equation}
\|G(v)\chi_{r_{2}}\|_{L^{\frac{2n}{n-2}}}\leq C\frac{R^{\frac{2\rho}{n\beta_{0}}}\beta}{r_{1}-r_{2}}\|G(v)\chi_{r_{1}}\|_{L^{\frac{2n\delta}{n-2}}}.
\end{equation}
Here $\chi_{r_{i}}$ is the characteristic function of $B_{1}(x_{0},r_{i})$.
\medskip

By taking $\beta=\delta^{-m}$ and $r_{m}=\frac{r_{1}(2+2^{-m})}{4}$, the standard Moser iteration shows that
$$\|v\|_{L^{\infty}(B_{1}(x_{0},\frac{r_{1}}{2}))}\leq CR^{\frac{2\rho\delta}{n\beta_{0}(1-\delta)}}=CR^{\frac{(n-2)\rho}{n(\beta_{0}-1)}}.$$
Note that, if $\frac{(n-2)\rho}{n(\beta_{0}-1)}< n-2$, i.e., $\rho<n(\beta_{0}-1)$, there holds
$$ u_{j}(x_{0})=\frac{v_{j}(x_{0})}{R^{n-2}}\leq Cd^{-\alpha}(x_{0}),$$
where $$\alpha=\frac{n-2}{2}-\frac{(n-2)\rho}{2n(\beta_{0}-1)}.$$

From the above arguments, we know that $\rho_{0}$ can be chosen as $n(\beta_{0}-1)$. Here, $\beta_{0}$ should be chosen such that \eqref{eq6} holds as well as $\beta_{0}<\frac{n}{n-2}$. So, by a simple computation, we have $$\rho_{0} = \min\left(n\sqrt{\frac{Y_{\infty}}{Y(M)}}-n,\, \frac{2n}{n-2}\right).$$
\medskip

\noindent {\bf Case 2:} $Y(M)< 0$. In this case, we have $Y_{j}\leq 0$ when $j$ sufficiently large, since $\{Y_{j}\}$ converges decreasingly to $Y(M)$. Thus we can directly drop the last term in \eqref{maineq}. Then \eqref{eq8} turns out to be
\begin{equation}
\|G(v)\eta\|_{L^{p}}^{2}\leq CC_{0}\beta^{2}\|\nabla_{1}\eta\|_{L^{\infty}}^{2}\int G(v)^{2}\ d\v_{\widetilde{g}}.
\end{equation}
For the present situation, we may directly choose $\rho_{0}=\frac{2n}{n-2}$ and take the same argument as in Case 1 to deduce
 $$\lim_{d(x)\rightarrow \infty}\lim_{j\rightarrow \infty}d(x)^{\alpha}u_{j}(x)=O(1)$$
where $\alpha=\frac{n-2}{4n}(2n-\rho(n-2))$. Thus we complete the proof.
\end{proof}

\subsection{Step 3.}\label{step3} Now we turn to showing $\{u_{j}\}$ is uniformly bounded with respect to $j$. For this purpose, we prove it by contradiction. If not, then there exists a subsequence $\{k\}\subseteq\{j\}$, $z_{k}\in M$, such that
$$u_{k}(z_{k})=\max u_{k}\triangleq m_{k}\rightarrow +\infty.$$
By Lemma \ref{lem3.1}, we know there exists a sufficiently large $R_{0}$ such that $z_{k}\in B_{R_{0}}(O)$ . Thus we can assume $z_{k}\rightarrow z_{0}$. Take a normal coordinate system at $z_{0}$. It is well-known that, in the normal coordinate system, we have
$$ g_{ij}(x)=\delta_{ij}+O(|x|^{2}),\s\s \mbox{and}\s\s det(g_{ij}(x))=1+O(|x|^{2}).$$
Denote the coordinates of $z_{k}$ at this atlas by $x_{k}$. Then, $x_{k}\rightarrow 0$ as $k\rightarrow \infty$. With respect to this coordinate chart, $u_{k}$ satisfies the following equation:
\begin{equation}
\frac{1}{\sqrt{\det g}}\partial_{i}(\sqrt{\det g}g^{ij}\partial_{j}u_{k}) -c(n)R_{g}(x)u_{k}+Y_{k}u_{k}^{p-1}=0
\end{equation}
Without loss of generality, we may assume the above equation can be defined in $\{x:|x|<1\}$. Now define
$$v_{k}=m_{k}^{-1}u_{k}(\delta_{k}x+x_{k})$$
where $\delta_{k}=m_{k}^{1-p/2}\rightarrow 0.$ Then $v_{k}$ can be defined on the ball centered at $0$ with radius $\rho_{k}=(1-|x_{k}|)/\delta_{k}\rightarrow \infty$ in $\mathbb{R}^{n}$. Moreover, $v_{k}$ satisfying the following equation
\begin{equation}\label{equa1}
\frac{1}{b_{k}}\partial_{i}(b_{k}a^{ij}_{k}\partial_{j}v_{k})-c_{k}v_{k}+Y_{k}v_{k}^{p-1}=0,
\end{equation}
where
\begin{eqnarray}
a_{k}^{ij}(x)&=&g^{ij}(\delta_{k}x+x_{k})\rightarrow \delta_{ij},\label{eq20}\\
b_{k}(x)&=&\sqrt{\det g(\delta_{k}x+x_{k})}\rightarrow 1,\label{eq21}\\
c_{k}(x)&=& c(n)m_{k}^{2-p}R_{g}(\delta_{k}x+x_{k})\rightarrow 0.\label{eq22}
\end{eqnarray}
The above convergence is actually $C^{1}$ uniform convergence on any finite domain of $\mathbb{R}^{n}$. Noting that
$$0\leq v_{k}\leq v_{k}(0)=1.$$
By $L^{p}$ and Schauder estimate we obtain that, for any $R>0,$ there exists $C(R)>0$ and $k(R)>0$, such that
$$\|v_{k}\|_{C^{2,\alpha}(\overline{B}_{R})}\leq C(R),\ \ \ \ \forall \ k\geq k(R).$$
Picking $R_{m}\rightarrow +\infty$, we take a standard diagonal argument to know that there exists a subsequence $\{v_{m}\}$, such that $v_{m}\rightarrow v\in C^{2}(\mathbb{R}^{n})$ with respect to $C^{2}$-norm on every $\overline{B}_{R_{m}}$. Let $m\rightarrow \infty$, in view of \eqref{equa1}, \eqref{eq20}, \eqref{eq21} and \eqref{eq22} we know that $v$ is a nonnegative solution of the following equation
\begin{equation}\label{equa3}
\Delta v+Y(M)v^{p-1}=0,
\end{equation}
with $v(0)=1.$ By the maximal principle, we have $v>0$.

By changing of the variables, we obtain
\begin{eqnarray}
\int_{|x|\leq\frac{1}{2}\delta_{k}^{-1}}v_{k}^{p}b_{k}\ dx=\int_{B_{\frac{1}{2}}(x_{k})}u_{k}^{p}\sqrt{\det{g}}\ dx\leq\|u_{k}\|_{L^{p}}^{p}=1.
\end{eqnarray}

Since $\{v_{m}^{p}b_{m}\}$ converges to $v^{p}$ uniformly on any bounded domain in $\mathbb{R}^{n}$, by Fatou's lemma we obtain
\begin{equation}\label{ineq1}
\int_{\mathbb{R}^{n}}v^{p}dx\leq 1.
\end{equation}
Similarly, we have
\begin{equation}\label{s5}
\int_{|x|\leq\frac{1}{2}\delta_{k}^{-1}}|\nabla v_{k}|^{2}b_{k}\ dx=\int_{B_{\frac{1}{2}}(x_{k})}|\nabla u_{k}|^{2}\sqrt{\det{g}}\ dx\leq\|\nabla u_{k}\|_{L^{2}(B_{R_{0}}(O))}^{2}.
\end{equation}
By $L^{p}$ estimate, we have
\begin{equation}\nonumber
\| u_{k}\|_{W^{2,\frac{2n}{n+2}}(B_{R_{0}}(O))}\leq C.
\end{equation}
Then the Sobolev embedding theorem yields
\begin{equation}\nonumber
\|\nabla u_{k}\|_{L^{2}(B_{R_{0}}(O))}\leq C.
\end{equation}
Combining the above inequality and \eqref{s5}, by Fatou's lemma again we obtain
\begin{equation}
\int_{\mathbb{R}^{n}}|\nabla v|^{2}dx<\infty.
\end{equation}
Choose $\eta\in C_{0}^{\infty}(\mathbb{R}^{n})$ such that $0\leq \eta\leq1$, $\eta=1$ when $x\in B_{1}(0)$, $\eta=0$ when $x\in \mathbb{R}^{n}\backslash B_{2}(0)$. Define $$v_{R}(x)=\eta(\frac{x}{R})v(x).$$
Then, obviously we have
\begin{equation}\label{con1}
\int_{\mathbb{R}^{n}}(|\nabla (v-v_{R})|^{2}+|v-v_{R}|^{p})\ dx\rightarrow 0,\s\s \mbox{as} \s R\rightarrow \infty.
\end{equation}
Multiplying \eqref{equa3} by $v_{R}$ and integrating by parts, we get
\begin{equation}\label{con2}
\int_{\mathbb{R}^{n}}\nabla v\cdot \nabla v_{R}\ dx=Y(M)\int_{\mathbb{R}^{n}}v^{p-1}v_{R}\ dx.
\end{equation}
In view of \eqref{con1}, we let $R\rightarrow \infty$ in \eqref{con2} to obtain
\begin{equation}\label{con3}
\int_{\mathbb{R}^{n}}|\nabla v|^{2}\ dx=Y(M)\int_{\mathbb{R}^{n}}v^{p}\ dx.
\end{equation}
Now, by virtue of \eqref{ineq1}, \eqref{con3} and the Sobolev inequality we get
\begin{equation}\label{con4}
\Lambda\left(\int_{\mathbb{R}^{n}}v^{p}\ dx\right)^{\frac{2}{p}}\leq\int_{\mathbb{R}^{n}}|\nabla v|^{2}\ dx=Y(M)\int_{\mathbb{R}^{n}}v^{p}\ dx.
\end{equation}
So we have
\begin{equation}\label{con5}
\Lambda\leq Y(M)\left(\int_{\mathbb{R}^{n}}v^{p}\ dx\right)^{\frac{2}{n}}\leq Y(M).
\end{equation}
This contradicts the assumption $Y(M)<\Lambda.$
\medskip

\subsection{Step 4.}The convergence of $\{u_{j}\}$.

By the standard elliptic theory, $u_{j}$ is uniformly bounded in $C^{k,\alpha}$, $\forall k\in \mathbb{N}$. Hence, there exists a subsequence of $\{u_j\}$ which converges to $u$ satisfying
$$\Delta u-c(n)R_{g}u+Y(M)u^{p-1}=0.$$
However, we do not know whether or not $u\neq0$. The next theorem tells us that, under the hypothesis $Y(M)<Y_{\infty}(M)$, there holds $u\not\equiv0$.
\begin{pro}\label{concentrate}
Assume that $u$ is the limit function of $\{u_{j}\}$ as above. If $Y(M)<Y_{\infty}(M)$, then $u\not\equiv0$.
\end{pro}

\begin{proof}
We prove this lemma by contradiction. If $u\equiv0$, then we know $u_{j}$ converge to $0$ on any compact set of $M$. For any fixed $R$, let $\eta(r)$ be a smooth function such that $\eta(r)=1$ when $r\geq2R$; $\eta(r)=0$ when $r\leq \frac{3}{2}R$. Then we have
\begin{eqnarray}\label{eq15}
& &\int_{M}(|\nabla u_{j}|^{2}+c(n)R_{g}u_{j}^{2})\ d\v_{g}\nonumber\\&=&\int_{M\backslash B_{R}}(|\nabla u_{j}|^{2}+c(n)R_{g}u_{j}^{2})\ d\v_{g}+\int_{ B_{R}}(|\nabla u_{j}|^{2}+c(n)R_{g}u_{j}^{2})\ d\v_{g}\nonumber\\
&=&\int_{M\backslash B_{R}}(|\nabla \eta u_{j}|^{2}+c(n)R_{g}\eta^{2}u_{j}^{2})\ d\v_{g}+\int_{B_{2R}\backslash B_{R}}(|\nabla u_{j}|^{2}+c(n)R_{g}u_{j}^{2})\ d\v_{g}\nonumber\\
& &-\int_{B_{2R}\backslash B_{R}}(|\nabla \eta u_{j}|^{2}+c(n)R_{g}\eta^{2}u_{j}^{2})\ d\v_{g}+\int_{ B_{R}}(|\nabla u_{j}|^{2}+c(n)R_{g}u_{j}^{2})\ d\v_{g}\nonumber\\
&\geq& Y\left(M\backslash B_{R}\right)\left(\int_{M\backslash B_{R}}|\eta u_{j}|^{p}\ d\v_{g}\right)^{\frac{2}{p}}+\int_{B_{2R}\backslash B_{R}}(|\nabla u_{j}|^{2}+c(n)R_{g}u_{j}^{2})d\v_{g}\nonumber\\
& &-\int_{B_{2R}\backslash B_{R}}(|\nabla \eta u_{j}|^{2}+c(n)R_{g}\eta^{2}u_{j}^{2})\ d\v_{g}+\int_{ B_{R}}(|\nabla u_{j}|^{2}+c(n)R_{g}u_{j}^{2})d\v_{g}\label{eq15},
\end{eqnarray}
and
\begin{equation}\label{eq16}
\int_{M\backslash B_{R}}|\eta u_{j}|^{p}\ d\v_{g}=1-\int_{B_{2R}}u_{j}^{p}\ d\v_{g}+\int_{B_{2R}\backslash B_{R}}(\eta u_{j})^{p}\ d\v_{g}.
\end{equation}
Substitute \eqref{eq16} to \eqref{eq15}, then let $j\rightarrow\infty$ to get
\begin{equation}
Y(M)\geq Y\left(M\backslash B_{R}\right)\nonumber.
\end{equation}
Since the above inequality holds for any fixed $R$, let $R\rightarrow\infty$, we obtain
\begin{equation}
Y(M)\geq Y_{\infty}(M).
\end{equation}
So we have
$$Y(M)= Y_{\infty}(M)$$
which contradicts to the hypothesis.
\end{proof}

By Proposition \ref{concentrate} we know $u\not\equiv0$. Using maximal principle, we obtain $u(x)>0$.  Now after a suitable dilation, we can obtain a positive solution to \eqref{Yamabe equation} with $K=1,0,-1$ when $Y(M)$ is positive, $0$ and negative respectively.
This completes the proof.
\\medskip

\section{ Proof of Theorem \ref{mainthm2} and Theorem \ref{mainthm3}}\label{proof2}
In this section, we will study the blowup behavior of $\{u_{i}\}$ under the pointed Cheeger-Gromov topology. First of all, we prove a uniform estimate of $u_{j}$ near the boundary. The method used here is the Giorgi-Nash-Moser iteration just as in the argument in the step 2 of the above section.\medskip

Let $u_{j}$ be the positive solution obtained in Proposition \ref{solution} and $$U_{j}= \{x\in B_{j}(O)| d\left(x,\partial B_{j}(O)\right)<\frac{1}{8}\}.$$
In order to prove Theorem \ref{mainthm2} and \ref{mainthm3}, we need to establish the following theorem.
\begin{thm}\label{boundary}
There exists a positive constant $C$ which does not depend on $j$ such that $u_{j}(x)\leq C$ for all $x\in U_{j}$ when $j$ is large enough.
\end{thm}
\proof
Extend $u_{j}$ to the whole manifold by defining $u_{j}(x)=0$ when $x\notin B_{j}(O)$. The extended function, still denoted it by $u_{j}$, is continuous and satisfies
\begin{equation}\label{eq17}
\Delta u_{j}-c(n)R_{g}u_{j}+Y_{j}u_{j}^{p-1}\geq 0\s\s \mbox{on} \s M.
\end{equation}
Let $G(s)=s^{\beta}$ and define
$$ F(t)=\int_{0}^{t}G'(s)^{2}ds=\frac{\beta^{2}}{2\beta-1}t^{2\beta-1}.$$
By a simple computation, we have that, as $\beta>1$, there holds true
\begin{equation}
sF(s)\leq s^{2}G'(s)^{2}=\beta^{2}G(s)^{2} .
\end{equation}

Take $\phi\in C^{\infty}[0,\infty)$ such that $0\leq\phi\leq1$; $\phi(r)=1$, when $r\in [0,1/2]$; $\phi(r)=0$, when $r\in [1,\infty)$; and $|\phi'(r)|\leq C$. For any fixed $x_{j}\in \partial B_{j}(O)$, let $\eta(x)=\phi(d(x_{j},x))$. Obviously, $\mbox{spt}(\eta^{2}F(v))\subseteq M\setminus B_{\frac{j}{2}}(O)$. Multiplying the both side of \eqref{eq17} by $\eta^{2}F(v)$ and integrating by parts yields that, for some $\epsilon>0$,
\begin{align*}
 &\|\nabla(G(u_{j})\eta)\|_{L^{2}}^{2}+\int c(n)R_{g}(G(u_{j})\eta)^{2}d\v_{g}\\
\leq & C\beta^{2}\|\nabla\eta\|_{L^{\infty}}^{2}\int G(u_{j})^{2}d\v_{g}+(\beta^{2}+\epsilon)Y_{j}\int u_{j}^{p-2}G(u_{j})^{2}\eta^{2}d\v_{g},
\end{align*}
where $\epsilon $ can be chosen arbitrary small.

Since $Y\left(M\backslash B_{r}(O)\right)$ is increasing with respect to $r$, we infer from the assumption $Y_{\infty}(M)>0$ that $Y\left(M\backslash B_{r}(O)\right)>0$ when $r$ is large. Let
$$C_{j}=\frac{1}{Y\left(M\backslash B_{\frac{j}{2}}(O)\right)}.$$
Noting $\mbox{spt}(\eta^{2}F(u_{j}))\subseteq M\setminus B_{\frac{j}{2}}(O)$, by the definition of the Yamabe quotient we have
$$\|G(u_{j})\eta\|_{L^{p}}^{2}\leq C_{j}\|\nabla(G(u_{j})\eta)\|_{L^{2}}^{2}+C_{j}\int_{B_{1}(x_{j})}c(n)R_{g}[G(u_{j})\eta]^{2}d\v_{g}.$$
By a similar argument with that in the step 2 of Section 3, we derive from the above inequality
\begin{eqnarray}
\|G(u_{j})\eta\|_{L^{p}}^{2}
&\leq& CC_{j}\beta^{2}\|\nabla\eta\|_{L^{\infty}}^{2}\int G(u_{j})^{2}d\v_{g}+C_{j}Y_{j}(\beta^{2}+\epsilon)\int u_{j}^{p-2}G(u_{j})^{2}\eta^{2}d\v_{g}\\
&\leq& CC_{j}\beta^{2}\|\nabla\eta\|_{L^{\infty}}^{2}\int G(u_{j})^{2}d\v_{g}+C_{j}Y_{j}(\beta^{2}+\epsilon)\|G(u_{j})\eta\|_{L^{p}}^{2}\label{eq18}.
\end{eqnarray}
By the hypothesis, we have
$$\lim_{j\rightarrow +\infty}C_{j}Y_{j}=\frac{Y(M)}{Y_{\infty}(M)}<1.$$
Hence, we can choose $\beta_{0}$ sufficiently close to $1$, $j$ sufficiently large and $\epsilon$ sufficiently small such that
\begin{equation}\label{eq19}
(\beta_{0}^{2}+\epsilon)C_{j}Y_{j}\leq \lambda<1.
\end{equation}
Substituting \eqref{eq19} into \eqref{eq18} leads to
\begin{equation}
\|u_{j}^{\beta_{0}}\eta\|_{L^{p}}^{2}\leq C\int|\nabla\eta|^{2}u_{j}^{2\beta_{0}}d\v_{g}.
\end{equation}
By the definition of $\eta$ and H\"{o}lder inequality
\begin{eqnarray}
\left(\int_{B_{\frac{1}{2}}(x_{j})}u_{j}^{\frac{2n\beta_{0}}{n-2}}\ dvol_{g}\right)^{\frac{n-2}{n}}&\leq&\|u_{j}^{\beta_{0}}\eta\|_{L^{p}}^{2}\leq C\int_{B_{1}(x_{j})}u_{j}^{2\beta_{0}}\ d\v_{g}\nonumber\\
&\leq& C(\v_{g}(B_{1}(x_{j}))^{\frac{n-(n-2)\beta_{0}}{n}}\nonumber\leq C.
\end{eqnarray}
Here we have used the volume comparison theorem in the above inequality. Then, by almost the same iteration argument as in the previous section we can show that
$$\|u_{j}\|_{L^{\infty}(B_{\frac{r_{1}}{2}}(x_{j}))}\leq C,$$
where $r_{1}<\frac{1}{2}$ is any fixed positive number. Thus we complete the proof.
\endproof
\medskip

\noindent {\bf Proof of Theorem \ref{mainthm2}.} It is sufficient to show that $\{u_{j}\}$ is uniformly bounded on any given compact set on $M$. If not, then the following two situations appear.
\medskip

\noindent {\bf Case 1.} $\{u_{j}\}$ blow up at `interior' of $M$, i.e. there exists a subsequence $\{k\}\subseteq\{j\}$, $z_{k}\in M$, such that
$$u_{k}(z_{k})=\max u_{k}\triangleq m_{k}\rightarrow +\infty,$$
where $z_{k}\in K$ and $K\subseteq M$ is a compact subset. By the same arguments as in the step 3 of the previous section, we know this is impossible.

\medskip

\noindent {\bf Case 2.} $\{u_{j}\}$ blow up at `infinity' of $M$, i.e there exists a subsequence $\{k\}\subseteq\{j\}$, $z_{k}\in M$, such that
$$u_{k}(z_{k})=\max u_{k}\triangleq m_{k}\rightarrow +\infty,$$
where $z_{k}\rightarrow \infty$. If this case occurs, we can not choose a normal coordinate system at ``infinity'' just as in Case $1$. To overcome this difficulty, we consider the sequence of pointed manifold $(M, z_{i},g)$.  By Theorem $1.1$ and  Remark $2.4$ in \cite{Anderson2}, we know there exists a subsequence denoted by $\{z_{j}\}$ such that $\{(M, z_{i},g)\}$ converges in the $C^{1,\alpha}$ topology to a complete pointed Riemann manifold $(M_{\infty},z_{\infty}, g_{\infty})$ under the assumption $|Ric|\leq c$ and $\mbox{inj}(M)\geq a>0$.
\medskip

Take a normal coordinate system $\{x_{i}\}$ around $z_{\infty}$ on $M_{\infty}$. Without loss of generality, we can assume this coordinate chart is defined on $B_{\frac{1}{16}}(z_{\infty})$. By the definition \ref{cheeger}, we know there exist $F_{i}$ such that $F_{i}(z_{\infty})=z_{i}$, $F_{i}({B_{\frac{1}{16}}(z_{\infty})})\subset B_{\frac{1}{8}}(z_{i})$ when $i$ is sufficiently large, and $F_{i}^{*}g\rightarrow g_{\infty}$ on $B_{\frac{1}{16}}(z_{\infty})$ in the $C^{1,\alpha}$ topology. Moreover, by Theorem \ref{boundary}, we have $B_{\frac{1}{8}}(z_{i}) \subset B_{i}(O)$ when $i$ is large enough. Denote
\begin{equation}
F_{i}^{*}g=g_{i}\s\s \mbox{and}\s\s v_{j}=u_{j}\comp F_{j}.
\end{equation}
Then $v_{j}$ satisfies the following equation on $B_{\frac{1}{16}}(z_{\infty})$
\begin{equation}
\Delta_{g_{j}}v_{j}-c(n)(R_{g}\comp F_{j})v_{j}+Y_{j}v_{j}^{p-1}=0.
\end{equation}
Let $\widetilde{v}_{k}=m_{k}^{-1}v_{k}(\delta_{k}x)$, where $m_{k}$ and $\delta_{k}$ is the same as that in Case $1$. Obviously, the definition domain of $\widetilde {v}_{j}$ is the ball centered at $0$ with radius $\rho_{k}=\frac{1}{16\delta_{k}}\rightarrow \infty$ in $\mathbb{R}^{n}$. Moreover $\widetilde{v}_{k}$ satisfies the following equation
\begin{equation}\label{equa2}
\frac{1}{\tilde{b}_{k}}\partial_{i}(\tilde{b}_{k}\tilde{a}^{ij}_{k}\partial_{j}\widetilde{v}_{k})-\tilde{c}_{k}\widetilde{v}_{k}+Y_{k}\widetilde{v}_{k}^{p-1}=0,
\end{equation}
where
\begin{eqnarray}
\tilde{a}_{k}^{ij}(x)&=&g^{ij}_{k}(\delta_{k}x)\rightarrow \delta_{ij},\label{eq23}\\
\tilde{b}_{k}(x)&=&\sqrt{\det g_{k}(\delta_{k}x)}\rightarrow 1,\label{eq24}\\
\tilde{c}_{k}(x)&=& c(n)m_{k}^{2-p}(R_{g}\comp F_{k})(\delta_{k}x)\rightarrow 0\label{eq25}.
\end{eqnarray}
Here,
\begin{equation}
0\leq \widetilde{v}_{k}\leq \widetilde{v}_{k}(0) = 1.
\end{equation}
By $L^{p}$ estimate we obtain that, for any $R>0$ and $q>0$, there exists $C(R)>0$ and $k(R)>0$ such that
$$\|\widetilde{v}_{k}\|_{W^{2,q}(B_{R})}\leq C(R),\ \ \ \ \forall \ k\geq k(R).$$
The Sobolev embedding theorem yields $\widetilde{v}_{k}\in C^{1,\alpha}(B_{R}(0))$. Hence, by taking a subsequence we get
$$\widetilde{v}_{k}\rightarrow v\s\s\mbox{in} \s\s C^{1,\alpha}.$$
It is easy to see that $v$ is a $C^{1,\alpha}$ weak solution of the following equation
$$\Delta v+Y(M)v^{p-1}=0\s\s \mbox{in}\s \s B_{R}(0).$$
Since $v\leq1$, the standard elliptic theory tells us that $v$ is smooth.
\medskip

Choosing $R_{m}\rightarrow +\infty$ and taking a standard diagonal argument we know that there exists a subsequence $\{v_{m}\}$ such that $v_{m}\rightarrow v$ with respect to the $C^{1,\alpha}$ norm on every compact subset in $\R^{n}$. Letting $m\rightarrow \infty$, in view of \eqref{equa2}, \eqref{eq23}, \eqref{eq24} and \eqref{eq25} we know that $v$ is a nonnegative solution with $v(0)=1$ of the following equation
\begin{equation}
\Delta v+Y(M)v^{p-1}=0.
\end{equation}
The maximal principle yields $v>0$.

Similarly, we have
$$\int_{\mathbb{R}^{n}} v^{p} dx\leq 1\s\s\s \mbox{and} \s\s\s \int_{\mathbb{R}^{n}}|\nabla v|^{2}dx<\infty.$$
For the present situation, it is easy to see that $v$ also satisfies \eqref{con1}, \eqref{con2}, \eqref{con3}, \eqref{con4} and \eqref{con5}. Thus we can get the same contradiction.
\medskip

From now on we know $\{u_{j}\}$ is uniformly bounded on $M$. By the standard elliptic theory, $u_{j}$ is uniformly bounded in $C^{2,\alpha}$ on any compact set $K\subseteq M$. Hence, there exists a subsequence which converges to $u$ satisfying
$$\Delta u-c(n)R_{g}u+Y(M)u^{p-1}=0.$$
By Proposition \ref{concentrate} again, we know $u$ is a positive solution. Then, by a suitable scaling we can obtain a positive solution to \eqref{Yamabe equation} with $K=1,0,-1$ when $Y(M)$ is positive, $0$ and negative respectively. Thus we complete the proof.
\endproof\medskip

\noindent {\bf Proof of Theorem \ref{mainthm3}.} By Theorem $2.6$ in \cite{Anderson2}, we know pointed manifolds $(M_{i}, z_{i}, g_{i})$ satisfying the condition $(i)$ in Theorem \ref{mainthm3} will converge in the pointed Gromov-Hausdorff topology, to an $n$ dimensional orbifold $(V,g)$ with finite number of singular points, each having a neighborhood homeomorphic to the cone $C(S^{n-1}/\Gamma)$ with $\Gamma$ a finite subgroup of $O(n)$. Furthermore, this convergence is $C^{1,\alpha}$ off the singular points. However, if these manifolds satisfy the additional condition $(ii)$ in Theorem \ref{mainthm3}, then the singularities of this orbifold do not arise, see Theorem $A'$ in \cite{Anderson1}, Remark $2.7$ and Corollary $2.8$ in \cite{Anderson2}. All in all, we have $(M, z_{i},g)$ converge in the $C^{1,\alpha}$ topology to a complete pointed Riemann manifold $(M_{\infty},z_{\infty}, g_{\infty})$. The proof of this Theorem is exactly the same as the proof of Theorem \ref{mainthm2}.\endproof

\bigskip

\noindent {\bf Acknowledgment}:
The author would like to thank his supervisors Professor Youde Wang and Professor Yuxiang Li for their encouragement and inspiring advices.

{}

\vspace{1.0cm}

Guodong Wei

{\small\it Academy of Mathematics and Systems Sciences, Chinese Academy of Sciences, Beijing 100080,  P.R. China.}

{\small\it Email: weiguodong@amss.ac.cn}

\end{document}